\documentclass[graybox]{svmult}

\usepackage{type1cm}        
\usepackage{makeidx}         
\usepackage{graphicx}        
\usepackage{multicol}        
\usepackage[bottom]{footmisc}

\usepackage{newtxtext}       %
\usepackage[varvw]{newtxmath}       

\makeindex             

\begin{document}

\title*{Discontinuous Galerkin method for linear wave equations involving derivatives of the Dirac delta distribution}
\titlerunning{DG method for wave equations involving derivatives of the Dirac delta distribution}
\author{Scott E. Field \and Sigal Gottlieb \and Gaurav Khanna \and Ed McClain}

\institute{Scott E. Field \at
          Department of Mathematics, 
          Center for Scientific Computing \& Visualization Research, 
          University of Massachusetts, 
          Dartmouth, MA 02747. 
          \email{sfield@umassd.edu}
          \and
          Sigal Gottlieb \at
          Department of Mathematics, 
          Center for Scientific Computing \& Visualization Research, 
          University of Massachusetts, 
          Dartmouth, MA 02747. 
          \email{sgottlieb@umassd.edu}
          \and
          Gaurav Khanna \at
          Department of Physics, 
          Center for Scientific Computing \& Visualization Research, 
          University of Massachusetts,
          Dartmouth, MA 02747. \& 
          Department of Physics,
          University of Rhode Island,
          Kingston, RI 02881.
          \email{gkhanna@uri.edu}
          \and
          Ed McClain \at
          Department of Mathematics, 
          Center for Scientific Computing \& Visualization Research, 
          University of Massachusetts, 
          Dartmouth, MA 02747. 
          \email{emcclain@umassd.edu}
}

%

\maketitle

\vspace{-1in} 

\abstract{Linear wave equations sourced by a Dirac delta distribution $\delta(x)$ and its derivative(s) can serve as a model for many different phenomena. We describe a discontinuous Galerkin (DG) method to numerically solve such equations with source terms proportional to 
$\partial^n \delta /\partial x^n$. Despite the presence of singular source terms, which imply discontinuous or potentially singular solutions, our DG method achieves global spectral accuracy even at the source's location. Our DG method is developed for the wave equation written in fully first-order form. The first-order reduction is carried out using a distributional auxiliary variable that removes some of the source term's singular behavior. While this is helpful numerically, it gives rise to a distributional constraint. We show that a time-independent spurious solution can develop if the initial constraint violation is proportional to $\delta(x)$. Numerical experiments verify this behavior and our scheme's convergence properties by comparing against exact solutions.}


\section{Introduction}
\label{sec:Intro-dirac-field}

In this article we describe a discontinuous Galerkin (DG) 
method~\cite{Reed.W;Hill.T1973,Hesthaven2008,Cock01,cockburn1998runge,Cockburn.B1998,Cockburn.B;Karniadakis.G;Shu.C2000} for solving 
the wave equation
\begin{equation}
\label{eq:generic-dirac-field}
- \partial_t^2 \psi + \partial_x^2 \psi + V(x) \psi  = \sum_{n=0}^N a_n(t,x) \delta^{(n)}(x) \,, 
\end{equation}
where $x \in [a,b]$, $V$ is a potential, $\delta^{(n)}(x) = \partial_x^n \delta(x)$ is the $n^{\rm th}$ distributional derivative~\cite{friedlander1998introduction} of a Dirac delta distribution $\delta(x)$, and $a_n(t,x)$ are arbitrary (classical) functions. We let the functions $\psi_0(x) = \psi(0,x)$ and $\dot{\psi}_0(x) = \partial_t \psi(0,x)$ specify the initial data. Differential equations of the form~\eqref{eq:generic-dirac-field} arise when modeling phenomena driven by well-localized sources and have found applications as diverse 
as neuroscience~\cite{casti2002population,caceres2011analysis}, seismology~\cite{petersson2009stable,shearer2019introduction}, and gravitational wave physics~\cite{sundararajan2007towards, poisson2011motion}. As one example, when a rotating blackhole is perturbed by a small, compact object the relevant (Teukolsky) equation features terms proportional to $\delta^{(2)}(x)$ on the right-hand-side~\cite{sundararajan2007towards}. To solve Eq.~\eqref{eq:generic-dirac-field}, various ``regularized"  numerical approaches~\cite{engquist2005discretization,tornberg2004numerical} and schemes~\cite{sundararajan2007towards,sundararajan2008towards,field2021gpu,canizares2010pseudospectral,canizares2009simulations,jung2007spectral,lousto2005time,sopuerta2006finite} have been proposed. Most of these methods only treat source terms proportional to $\delta(x)$ and $\delta^{(1)}(x)$ and do not achieve spectral accuracy at $x=0$.

Discontinuous Galerkin methods are especially well suited for solving 
Eq.~\eqref{eq:generic-dirac-field} and, more broadly, problems with delta distributions.
Indeed, the solution's non-smoothness can be 
``hidden" at an interface between subdomains. Furthermore, the DG method
solves the weak form of the problem, a natural
setting for the delta distribution.

To the best of our knowledge, two distinct DG-based strategies have appeared in the literature for solving hyperbolic equations with $\delta$-singularities. Yang and Shu~\cite{yang2013discontinuous} show that when the source term features a Dirac delta distribution (but no distributional derivatives of them), by using $k^{\rm th}$ degree polynomials the error will converge in a negative-order norm. Post-processing techniques are then used to recover high-order accuracy in the $L_2$ norm so long as the solution is not required near the singularity. A different approach, and the one we follow here, is based on the observation that (i) the solution is smooth to the left and the right of the singularity and (ii) if the singularity is collocated with a subdomain interface then the effect of the Dirac delta distribution is to modify the numerical flux. This framework was originally proposed by Fan et al.~\cite{fan2008generalized} for the Schr{\"o}dinger equation sourced by a delta distribution $\delta(x)$ and later extended by Field et al.~\cite{field2009discontinuous} to solve Eq.~\eqref{eq:generic-dirac-field} with source terms proportional to $\delta(x)$ and $\delta^{(1)}(x)$. 
Building on previous work~\cite{fan2008generalized,field2009discontinuous},
the main contributions of our paper are 
(i) to show how the DG method proposed in Refs.~\cite{fan2008generalized,field2009discontinuous} can be readily extended to solve Eq.~\eqref{eq:generic-dirac-field} as well as 
(ii) to clarify the importance of satisfying a distributional constraint that arises when performing a first-reduction of this equation.
We also derive equations that directly relate the coefficients
$a_n(t,x)$ to the numerical flux modification rule.

To motivate the main idea, consider the advection equation
\begin{equation}
\label{eq:advection-dirac-field}
\partial_t \psi + \partial_x \psi = \cos(t) \delta^{(1)}(x) \,,
\end{equation}
whose inhomogeneous solution is
\begin{equation}
\label{eq:advectionSol1-dirac-field}
\psi(t,x) = \cos(t) \delta(x) + H(x)\sin(t-x) \,,
\end{equation}
where the Heaviside step function obeys $H(x)=0$ for $x<0$, and $H(x)=1$ for $x \ge 0$. 
Away from $x=0$ the solution is smooth, suggesting a spectrally-convergent basis
should be used. At $x=0$ the solution is both discontinuous and contains a
term proportional to $\delta(x)$. 

Our proposed DG scheme proceeds as follows.
First, we use a change of variable $\bar{\psi} = \psi - \cos(t) \delta(x)$ to 
remove the singular term and solve the modified equation,
\begin{equation}
\label{eq:advectionModified-dirac-field}
\partial_t \bar{\psi} + \partial_x \bar{\psi} = \sin(t) \delta(x) \,.
\end{equation}
We will show this is always possible,
works for singular source terms $\delta^{(n)}(x)$, and is easy to enact.
Next, we use a non-overlapping, multi-domain setup such that $x=0$
is one of interface locations; for example, using two subdomains we would have 
$\mathsf{D}^1 = [a,0]$ and $\mathsf{D}^2 = [0,b]$. 
In each element we expand $\bar{u}$ in a polynomial basis of degree $k$
denoting this numerical approximation $\bar{\psi}_h \in V_h$, where
\begin{equation}
V_h = \{ v : v |_{\mathsf{D}^j} \in {\cal P}^k (\mathsf{D}^j), j=1, \dots, N \} \,,
\end{equation}
and ${\cal P}^k (\mathsf{D}^j)$ denotes the space of polynomials of degree at most $k$
defined on subdomain $\mathsf{D}^j$.
In each element, we then follow the standard DG procedure by integrating each
elementwise residual against 
all test functions $v(x) \in V_h$. 
A key ingredient of any DG method is the choice of numerical flux, denoted here by $\psi^*$,
which couples neighboring subdomains. As we will show later on, 
the $\delta$ distribution's effect on the scheme is 
to modify the numerical flux 
according to $\psi^* \rightarrow \psi^* + \sin(t)$ at the left side of $\mathsf{D}^2$. 
This extra $\sin(t)$ term arises from the evaluation of the
integral,
\begin{equation}
\int_a^b \sin(t) \delta(x) v(x) dx = 
\int_{\mathsf{D}^1} \sin(t) \delta(x) v_1(x) dx + \int_{\mathsf{D}^2} \sin(t) \delta(x) v_2(x) dx = \sin(t) v_2(0) \,,
\end{equation}
in a manner that is consistent with the defining property of
a delta distribution, $\int f(x) \delta(x) dx = f(0)$, 
and the expectation that the delta distribution
should only affect the solution to the right of $x=0$ (ie upwinding).
Here $v_{1,2} = v |_{\mathsf{D}^{1,2}}$ denote the global test function
restricted to a subdomain.  
Fig.~\ref{Fig:advecCartoon-dirac-field} provides a schematic of this procedure.
The numerical solution to Eq.~\eqref{eq:advection-dirac-field} is then 
$\psi_h = \bar{\psi}_h + \cos(t) \delta(x)$.

\begin{figure*}[thb]
\begin{center}
\includegraphics[scale=0.35]{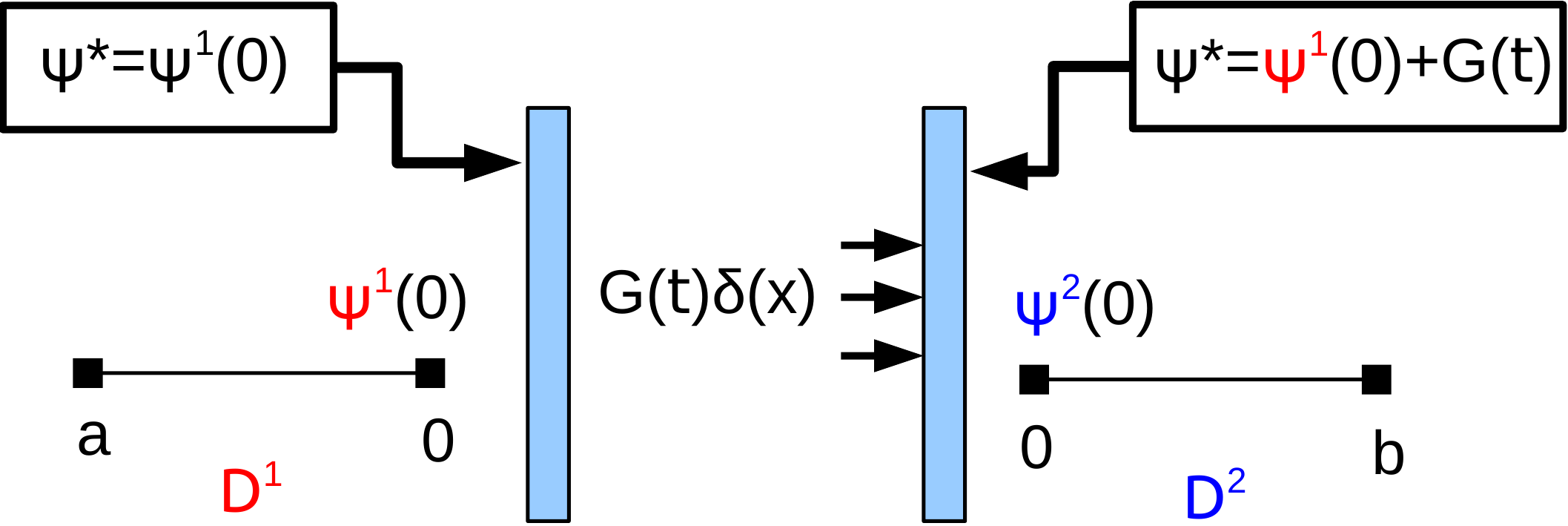}
\end{center}
\caption{Cartoon of the numerical flux modification enacted at $x=0$
to solve Eq.~\eqref{eq:advectionModified-dirac-field}. Here we show a generic
time-dependent source term $G(t)\delta(x)$ ``located between"
subdomain \textcolor{red}{$\mathsf{D}^{1}$} and \textcolor{blue}{$\mathsf{D}^{2}$} . Interfaces are represented schematically as
light blue rectangles.
The advection equation~\eqref{eq:advection-dirac-field} moves information from left-to-right, and accordingly
the upwinded numerical flux is modified to the right of the delta distribution. 
Sec.~\ref{sec:Intro-dirac-field} provides further discussion of the setup.}
\label{Fig:advecCartoon-dirac-field}
\end{figure*}

\begin{warning}{Remark}
When the coefficients $a_n(t,x)$ appearing in Eq.~\eqref{eq:generic-dirac-field} 
are functions of both independent variables, they can be put into the form assumed 
by Theorem~\eqref{thm:WaveMod-dirac-field} by the selection property of delta distributions. For 
example, $a_0(t,x) \delta(x)=a_0(t,0) \delta(x)$ 
or $a_1(t,x) \delta{}^{(1)}(x)= a_1(t,0) \delta{}^{(1)}(x) - \partial_x a_1(t,x) \rvert_{x=0}\delta{}(x)$.
Consequently, to streamline the discussion, we will exclusively focus on 
source terms of the form $\sum_{n=0}^N a_n(t) \delta^{(n)}(x)$ for the
remainder of this paper.
\end{warning}

\section{Reduction to a first-order system}
\label{sec:WaveModified-dirac-field}

Throughout, we use both an over--dot and superscript to denote $\partial / \partial t$
differentiation, for example 
$\partial a / \partial t = \dot{a}(t) = a^{(1)}(t)$,
and both a prime and superscript to denote $\partial / \partial x$, for example $\partial \delta / \partial x 
= \delta{}'(x) = \delta^{(1)}(x)$.

\subsection{Removing singular behavior}

The simple advection example introduced in Sec.~\ref{sec:Intro-dirac-field} demonstrates
how to remove a certain amount of singular behavior 
in the source term such that the new dependent variable, $\bar{\psi}$, 
contains no terms proportional to $\delta(x)$. The following theorem shows how to apply the
same procedure to Eq.~\eqref{eq:generic-dirac-field}.

\begin{theorem} \label{thm:WaveMod-dirac-field}
Consider Eq.~\eqref{eq:generic-dirac-field} with $V=0$.
Assume $a_{n}(t)$ has at least $n$ derivatives.
If $\psi$ is the exact solution to Eq.~\eqref{eq:generic-dirac-field}, then 
\begin{equation}
\label{eq:uBar-dirac-field}
\bar{\psi} = \psi - \sum_{i=0}^{\frac{N-1}{2}-1} \left[ 
\delta^{(2i)}(x) \sum_{n=0}^{\frac{N-1}{2} - 1 - i} a_{2n+2i+2}^{(2n)}(t)  + 
\delta^{(2i+1)}(x) \sum_{n=0}^{\frac{N-1}{2} - 1 - i} a_{2n+2i+3}^{(2n)}(t) \right] \,, 
\end{equation}
solves
\begin{equation}
\label{eq:generic-dirac-fieldMod-dirac-field}
- \partial_t^2 \bar{\psi} + \partial_x^2 \bar{\psi}  = G(t) \delta(x) + F(t) \delta{}'(x) \,,
\end{equation}
where
\begin{equation}
\label{eq:generic-dirac-fieldModSourceTerms-dirac-field}
G(t) = \sum_{n=0}^{(N-1)/2} a_{2n}^{(2n)}(t) \,, \qquad F(t) = \sum_{n=0}^{(N-1)/2} a_{2n+1}^{(2n)}(t) \,.
\end{equation}
\end{theorem}
\begin{proof}
Assume $N$ is odd, which is always possible by taking $a_N(t,x)=0$ if necessary.
Eq.~\eqref{eq:generic-dirac-field} can be transformed into Eq.~\eqref{eq:generic-dirac-fieldMod-dirac-field} by a sequence of 
$\left(N-1\right)/2$ substitutions of the form
\begin{equation*}
\bar{\psi}_{i+1} = \bar{\psi}_i - \sum_{n=0}^{N-2i-2} a_{n+2i+2}^{(2i)} (t) \delta^{(n)} (x) \,,
\end{equation*}
where we define $\bar{\psi}_{0} = \psi$. For example, we first substitute $\bar{\psi}_1$ into 
Eq.~\eqref{eq:generic-dirac-field}, generating a new wave equation for $\bar{\psi}_1$ with source
terms proportional to $\delta^{N(x)}$ and $\delta^{N-1}(x)$ removed. Next, we substitute
$\bar{\psi}_2$ into the PDE for $\bar{\psi}_1$, generating
a new wave equation for ${}^{(2)}\bar{\psi}$ with the source 
terms proportional to $\delta^{N-2}(x)$ and $\delta^{N-3}(x)$ removed. 
This process continues until we arrive at 
an equation for $\bar{\psi} = \bar{\psi}_{(N-1)/2}$
whose source terms are $G(t)\delta(x)$ and $F(t)\delta^{(1)}(x)$.
The final result can also be checked by direct computation. 
\end{proof}

\begin{warning}{Remark}
When $V \neq0$ we can still remove singular behavior from the source term, although we do not provide
a general expression as in Theorem~\ref{thm:WaveMod-dirac-field}. 
As a direct example, and assuming $V$ is differentiable, the differential equation
\begin{equation}
- \partial_t^2 \psi + \partial_x^2 \psi + V(x)\psi = 
a_0(t) \delta(x) + a_1(t) \delta^{(1)}(x) + a_2(t) \delta^{(2)}(x) + a_3(t) \delta^{(3)}(x) \,,
\end{equation}
can be transformed into 
\begin{equation} \label{eq:transformed_example-dirac-field}
- \partial_t^2 \bar{\psi} + \partial_x^2 \bar{\psi} + V(x) \bar{\psi} = 
\left[a_0 + a_2^{(2)} - a_2 V(0) + a_3 V^{(1)}(0) \right] \delta + 
\left[a_1 + a_3^{(2)} - a_3 V(0) \right] \delta{}^{(1)} \,,
\end{equation}
where $\bar{\psi} = \psi - a_2(t) \delta(x) - a_3(t) \delta{}^{(1)}(x) $. 
\end{warning}

\subsection{First-order reduction}
Thanks to Theorem~\ref{thm:WaveMod-dirac-field}, we can develop our DG method for the modified problem~\eqref{eq:generic-dirac-fieldMod-dirac-field}
and numerically solve for $\bar{\psi}$. At this point, we follow the approach of Ref.~\cite{field2009discontinuous}
and introduce the auxiliary variables,
\begin{equation}
\hat{\phi} = \partial_x{\bar{\psi}} \,, \qquad  \pi = - \partial_t \bar{\psi} \,,
\end{equation}
from which the original second-order wave equation~\eqref{eq:generic-dirac-field} can be rewritten 
as the following first-order system,
\begin{equation}
\begin{bmatrix}
\bar{\psi} \\
\pi \\
\hat{\phi}\\
\end{bmatrix}_{t}
+
\begin{bmatrix}
0 & 0 & 0 \\
0 & 0 & 1 \\
0 & 1 & 0\\ 
\end{bmatrix}
\begin{bmatrix}
\bar{\psi} \\
\pi \\
\hat{\phi} \\
\end{bmatrix}_{x}
+
\begin{bmatrix}
\pi \\
V \bar{\psi} \\
0 \\
\end{bmatrix}
= 
\begin{bmatrix}
0 \\
G(t) \delta(x) + F(t) \delta{}'(x) \\
0 \\
\end{bmatrix} \,.
\end{equation}
We find it convenient (this will be helpful in 
when deriving the numerical flux in Sec.~\ref{sec:nfulx-dirac-field}) 
to define a new distributional auxiliary variable
\begin{equation}
\phi = \hat{\phi} - F(t) \delta(x) \,,
\end{equation}
which removes all terms proportional to $\delta{}'(x)$,
\begin{equation} \label{eq:firstOrderv2-dirac-field}
\begin{bmatrix}
\bar{\psi} \\
\pi \\
\phi \\
\end{bmatrix}_{t}
+
\begin{bmatrix}
0 & 0 & 0 \\
0 & 0 & 1 \\
0 & 1 & 0\\ 
\end{bmatrix}
\begin{bmatrix}
\bar{\psi} \\
\pi \\
\phi \\
\end{bmatrix}_{x}
+
\begin{bmatrix}
\pi \\
V \bar{\psi} \\
0 \\
\end{bmatrix}
= 
\begin{bmatrix}
0 \\
G(t) \delta(x) \\
- \dot{F}(t) \delta{}(x)  \\
\end{bmatrix} \,.
\end{equation}
System~\eqref{eq:firstOrderv2-dirac-field} can be written more compactly as,
\begin{equation}
\label{eq:firstOrderv1-dirac-field}
\partial_t U + \partial_x F + \hat{V} = S(t) \delta(x) \,,
\end{equation}
for the system vector $U$, flux vector $F$, potential $\hat{V}$, and source vector $S$:
\begin{equation}
\begin{split}
U = \left[\bar{\psi}, \pi, \phi\right]^T \,, \quad & \quad
F(U) = \left[0, \phi, \pi \right]^T \,, \\ 
\hat{V} = \left[\pi, V \bar{\psi} , 0\right]^T \,,  \quad & \quad
S = \left[0, G(t) , - \dot{F}(t)\right]^T \,.
\end{split}
\end{equation}

\subsection{Distributional constraint}
\label{sec:AuxConstraints-dirac-field}

A solution to the first-order system~\eqref{eq:firstOrderv2-dirac-field} is also a solution to the original
PDE~\eqref{eq:generic-dirac-field} provided the distributional constraint
\begin{equation}
\label{eq:distributional_constraint-dirac-field}
C(t,x) = \phi - \partial_x{\bar{\psi}} + F(t) \delta(x) \,,
\end{equation}
vanishes. One can show that $C(t,x)$ obeys (upon setting $G=V=0$ for simplicity):
\begin{equation}
- \partial_t^2 C + \partial_x^2 C  = 0 \,.
\end{equation} 
Thus we conclude that if the initial data implies 
$C(0,x) = \dot{C}(0,x) = 0$, and our physical boundary condition is compatible with Eq.~\eqref{eq:distributional_constraint-dirac-field}, then $C = 0$ for all future times.

For certain applications, the exact initial data is unknown. If one is interested in the late-time behavior of the problem
due to the forcing term, trivial initial data ($\bar{\psi} = \phi = \pi = 0$) is often supplied instead. 
Trivial data results in an impulsive
(i. e. discontinuous in time) start-up, and a key question is if
a physical solution eventually emerges from such trivial initial data. The answer is unfortunately no.
Under the assumption of trivial initial data we have 
$C(0,x) = F(0) \delta(x)$ and $\dot{C}(0,x) = 0$
\footnote{The term $\dot{\phi}(0,x)$ is found from evaluation of the evolution 
equation~\eqref{eq:firstOrderv2-dirac-field}, $\dot{\phi} = -\pi{}'- \dot{F}(t) \delta{}(x)$, at $t=0$.}, giving
\begin{equation}
C(t,x) = \frac{1}{2} F(0) \left[\delta(x+t) + \delta(x-t)\right] \,.
\end{equation}
In numerical simulations, this manifests as a localized feature that advects off the computational grid.
We also observe a
constraint-violating spurious (or ``junk") solution develop in it's wake.
At arbitrarily late times, the physical and numerical solution will appear to differ by a time-independent offset. To see this, let $\bar{\psi}_{\rm exact}$ be the exact particular solution to Eq.~\eqref{eq:generic-dirac-fieldMod-dirac-field} and $\bar{\psi}_{\rm impulsive}$ be the solution to the first-order system~\eqref{eq:firstOrderv2-dirac-field} subject to trivial initial data. Then $\bar{\psi}_{\rm CV} = \bar{\psi}_{\rm impulsive} - \bar{\psi}_{\rm exact}$ is the spurious solution due to constraint violation. In Sec.~\ref{sec:experiments-dirac-field} we provide numerical evidence that
\begin{equation}
\bar{\psi}_{\rm CV} = \frac{1}{2} F(0) \left[H(x+t) H(-x) + H(x-t) H(x) \right] \,.
\end{equation}
is the spurious solution, as well as showing two ways to remove it.

\section{Discontinuous Galerkin Method}
\label{sec:DG-dirac-field}

To solve the wave equation~\eqref{eq:generic-dirac-field}
we first transform it into a simpler form using Theorem~\eqref{thm:WaveMod-dirac-field}
then carry out a fully first-order reduction. 
This section describes the nodal DG method we have implemented to numerically solve
the resulting system~\eqref{eq:firstOrderv1-dirac-field} subject to the constraint~\eqref{eq:distributional_constraint-dirac-field}.
The method is exactly the one first proposed in Ref.~\cite{field2009discontinuous}, and
so we only briefly summarize the key ideas. Indeed, a key contribution of our paper is
to show the methods of Ref.~\cite{field2009discontinuous} continue to be applicable 
for more challenging problems such as Eq.~\eqref{eq:generic-dirac-field}.

\subsection{The source-free method}

We divide the spatial domain into $N$ non-overlapping subdomains $a=x_0 < x_1 < \dots < x_N = b$
and denote $\mathsf{D}^j = [x_{j-1},x_j] $ as the $j^{\rm th}$ subdomain. In this one-dimensional setup,
the points $\{x_i\}_{i=1}^{N-1}$ locate the internal subdomain interfaces, and we require one of them to be $x=0$.
In each subdomain, 
each component of the vectors $U$, $F$, and $V$ are expanded in a polynomial basis, which are taken
to be degree-$k$ Lagrange interpolating polynomials $\{ \ell_i(x) \}_{i=0}^k$ defined from Legendre-Gauss-Lobatto nodes.
The time-dependent coefficients of this
expansion (e.g. on subdomain $j$: $\pi_h^j = \sum_{i=0}^k \pi_i(t) \ell_i(x)$) 
are the unknowns we solve for. We directly approximate $\bar{\psi}$, $\pi$, $\phi$, and $V$ and other terms 
arising in Eq.~\eqref{eq:firstOrderv1-dirac-field}, such as $\bar{\psi}V$, are achieved through pointwise products, for example
$\bar{\psi}_h V_h$.

On each subdomain, we follow the standard DG procedure by requiring the residual to satisfy 
\begin{equation}
\int_{\mathsf{D}^j} \left[\frac{ \partial U_h^j}{\partial_t} 
+ \frac{ \partial F_h^j}{\partial_x} + V_h^j\right] \ell_i^j dx  = 
\left[ \left(F_h^j - F^*\right) \ell_i^j \right]_{x_{j-1}}^{x^j}  \,,
\end{equation}
for all basis functions. Here we use an upwind numerical flux, $F^*(U_h)$, that depends on the values of the numerical solution 
from both sides of the interface.
These integrals can be pre-computed on a reference interval, leading to a coupled system of ordinary differential equations (see Eq.~47 of Ref.~\cite{field2009discontinuous}) that can be integrated in time. 

\subsection{Modifications to the numerical flux}
\label{sec:nfulx-dirac-field}

\begin{figure}[htb]
\begin{center}
\includegraphics[scale=0.35]{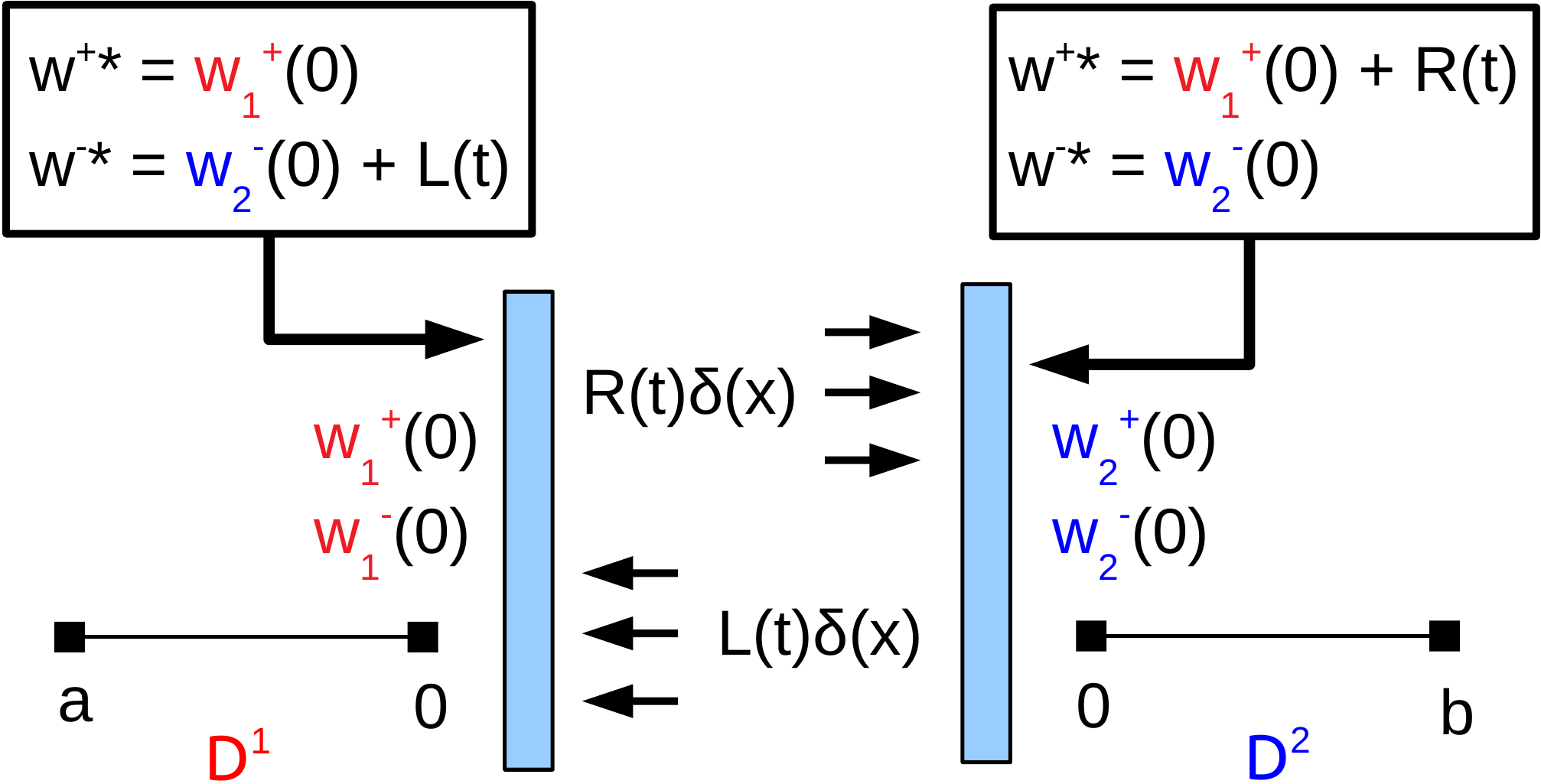}
\end{center}
\caption{Cartoon of the numerical flux modification enacted at $x=0$
to solve Eq.~\eqref{eq:generic-dirac-field} after its been
rewritten in terms of characteristic variables~\eqref{eq:advection_from_wave-dirac-field}.
Here we show a generic
time-dependent source term ``located between"
subdomain \textcolor{red}{$\mathsf{D}^{1}$} and \textcolor{blue}{$\mathsf{D}^{2}$}. 
Interfaces are represented schematically as light blue rectangles.
The system~\eqref{eq:advection_from_wave-dirac-field} includes two characteristic variables. The term
$R(t)\delta(x)$ sources the right moving wave and $L(t)\delta(x)$ sources the left moving wave.
The upwinded numerical flux is modified according to this directionality. 
Sec.~\ref{sec:nfulx-dirac-field} provides further discussion of the setup.}
\label{Fig:waveCartoon-dirac-field}
\end{figure}

With non-zero source terms, the numerical flux evaluated at the interface $x=0$ will be modified through additional terms. The form of these new terms were derived in Eq.~58 of Ref.~\cite{field2009discontinuous}. Instead of reproducing those results here, we provide an alternative viewpoint leading to the same result.

We begin by writing Eq.~\eqref{eq:firstOrderv1-dirac-field} in terms of characteristic variables,
\begin{equation} 
\begin{split}
    \partial_{t}w^+ + \partial_{x}w^+ + \frac{1}{2}V\bar{\psi} = \frac{1}{2}G(x)\delta(x) - \frac{1}{2}\dot{F}(t)\delta(x) = R(t) \delta(x) \label{eq:advection_from_wave-dirac-field} \\
    \partial_{t}w^- - \partial_{x}w^- + \frac{1}{2} V\bar{\psi} = \frac{1}{2}G(x)\delta(x) + \frac{1}{2}\dot{F}(t)\delta(x) = L(t) \delta(x)
\end{split}
\end{equation}
where $w^{\pm} = \left( \pi \pm \phi \right)/2$ and note that 
$\pi = w^{+} + w^{-}$ and $\phi = w^{+} - w^{-}$. 
We now have two copies of an advection equation with source terms proportional to $\delta(x)$. The equation 
for $w^+$ ($w^-$) describes a wave moving from left-to-right (right-to-left).
This allows us to apply the DG method described in the introduction: viewing the source term ``between" the two subdomains, we associate the entire contribution of $R(t)$
with the subdomain to the right and the entire contribution of $L(t)$ with the subdomain to the left. Fig.~\ref{Fig:waveCartoon-dirac-field} provides a schematic of this procedure for a simple 2 domain setup as well as the corresponding modification to the upwind numerical flux 
$(w^{\pm})^*$. Returning to the original system, the numerical flux is modified to the
left (right boundary point of $\mathsf{D}^{1}$) and right (left boundary point of $\mathsf{D}^{2}$) of $\delta(x)$ by
\begin{equation} 
F^*_{\rm left} \rightarrow F^*_{\rm left} + \left[0, -L(t), L(t) \right]^T \,, \qquad F^*_{\rm right} \rightarrow F^*_{\rm right} +  \left[0, R(t), R(t) \right]^T \,.
\end{equation}

\subsection{Defining the Heaviside function on the DG grid}
\label{sec:nfulx-dirac-field}

In our multi-domain setup, $x=0$ is both one of the interface locations and the location of the solution's discontinuity. When providing initial data, for example, we will sometimes need to evaluate the Heaviside function at $x=0$. Given the expected behavior of the solution, we will evaluate the Heaviside differently depending on the problem. Consider, for example, the advection equation, $\partial_t \psi + \partial_x \psi = \cos(t) \delta^{(1)}(x)$,  describing a right-moving wave and a two-subdomain setup, $\mathsf{D}^1 = [a,0]$ and $\mathsf{D}^2 = [0,b]$. In this case, we would evaluate the Heaviside according 
to $H(x)|_{\mathsf{D}^1} = 0$ and $H(x) |_{\mathsf{D}^2} = 1$. 
Now consider an advection equation describing a left-moving wave, 
$-\partial_t \psi + \partial_x \psi = \cos(t) \delta^{(1)}(x)$. For this problem, we would instead use 
$H(-x)|_{\mathsf{D}^1} = 1$ and $H(-x) |_{\mathsf{D}^2} = 0$.

\section{Distributional solutions to the 1+1 wave equation} \label{app:solutions-dirac-field}

This section presents exact solutions to the distributionally-forced $1+1$ wave equation. These solutions will be used in Sec.~\ref{sec:experiments-dirac-field} for testing our numerical scheme. 

Our recipe for solving
\begin{equation} \label{eq:ExactProblem-dirac-field}
-\partial_t^2 \Psi(t,x) + \partial_x^2 \Psi(t,x) = F(t) \delta^{(s)}(x) \, ,
\end{equation}
amounts to first solving 
\begin{equation} \label{eq:ModelPDE-dirac-field}
-\partial_t^2 \Psi(t,x;c,s) + \partial_x^2 \Psi(t,x;c,s) = F(t) \delta^{(s)}(x+c) \, ,
\end{equation}
for $s=0$ 
followed by an application of Eq.~\eqref{eq:generator-dirac-field}.
We shall view $c \in \mathbb{R}$ and $s \in \mathbb{Z}_{\geq 0}$ as parameters, and the solution $\Psi(t,x;c,s)$ as parameterized by them.

A solution to \eqref{eq:ModelPDE-dirac-field} can be found by the method of Green's function.
Recall the fundamental solution, $G(t,x;\widetilde{t},\widetilde{x})$, for
\begin{equation*}
-\partial_t^2 G + \partial_x^2 G = \delta(t-\widetilde{t})\delta(x-\widetilde{x}) \, ,
\end{equation*}
can be written in terms of the Heaviside~\cite{stakgold2011green}
\begin{equation*}
G(t,x;\widetilde{t},\widetilde{x}) = -\frac{1}{2} H(t-\widetilde{t} - |x-\widetilde{x}| ) \, .
\end{equation*}
Thus the solution to Eq.~\eqref{eq:ModelPDE-dirac-field} with $s=0$ can be written as 
\begin{align} \label{eq:PsiGreen-dirac-field}
\Psi(t,x;c,s=0) & = -\frac{1}{2} \int_0^t \int_{-\infty}^{\infty}  H(t-\widetilde{t} - |x-\widetilde{x}| ) F(\widetilde{t}) \delta(\widetilde{x}+c) d \widetilde{t} d \widetilde{x} \nonumber \\
& = -\frac{1}{2} \int_{0}^{t-|x+c|} F(t-|x+c| - y) d y \,,
\end{align}
where $y = t-\widetilde{t} - |x+c|$, and we have restricted to times $t \geq 0$ for which
the Heaviside is zero whenever $t - |x+c| < 0$. 
This $s=0$ solution generates an entire family of solutions corresponding to $s>0$. 
Clearly $\partial_c^s \Psi(t,x;c,s=0)$ solves Eq.~\eqref{eq:ModelPDE-dirac-field},
and so the particular solution of
\begin{equation*} 
-\partial_t^2 \Psi(t,x;0,s) + \partial_x^2 \Psi(t,x;0,s) = F(t) \delta^s(x) \, ,
\end{equation*}
is given by
\begin{equation} \label{eq:generator-dirac-field}
\Psi(t,x;0,s) = \partial_c^s \Psi(t,x;c,0) \bigg\rvert_{c=0} \,.
\end{equation}
We now provide explicit constructions for the cases considered in the numerical experiment section.

\begin{svgraybox} 
Let $F(t) = \cos(t)$ and $s=0$, then the generating function is
\begin{equation*}
\Psi_c(t,x;c,0) = 
- \frac{1}{2} \sin(t-|x+c|) \,,
\end{equation*}
and setting $c=0$ gives 
\[
\Psi(t,x) = - \frac{1}{2} \sin(t-|x|) \,.
\]
\vspace{-.175in}
\end{svgraybox}

\begin{svgraybox} 
Let $F(t) = \cos(t)$ and $s=1$, then the generating function is
\begin{equation*}
\partial_c \Psi_c(t,x;c,0) = \frac{1}{2} \mathrm{sgn}(x+c) \cos(t-|x+c|) \,,
\end{equation*}
and setting $c=0$ gives 
\begin{equation}
\label{eq:solution_s1-dirac-field}
\Psi(t,x) = \frac{1}{2} \mathrm{sgn}(x) \cos(t-|x|) \,.
\end{equation}
\vspace{-.175in}
\end{svgraybox}

\begin{svgraybox} 
Let $F(t) = \cos(t)$ and $s=2$, then the generating function is
\begin{align*}
\partial_c^2 \Psi_c(t,x;c,0) & = \delta(x+c)\cos(t-|x+c|) + \frac{1}{2} \mathrm{sgn}(x+c) \sin(t-|x+c|) \mathrm{sgn}(x+c) \\
& = \delta(x+c)\cos(t) + \frac{1}{2} \sin(t-|x+c|) \,,
\end{align*}
and setting $c=0$ gives 
\begin{equation}
\label{eq:solution_s2-dirac-field}
\Psi(t,x) = \delta(x)\cos(t) + \frac{1}{2} \sin(t-|x|) \,.
\end{equation}
\vspace{-.175in}
\end{svgraybox}

\section{Numerical results}
\label{sec:experiments-dirac-field}

\begin{figure}[!htb]
    \begin{center}
    \includegraphics[scale=0.38]{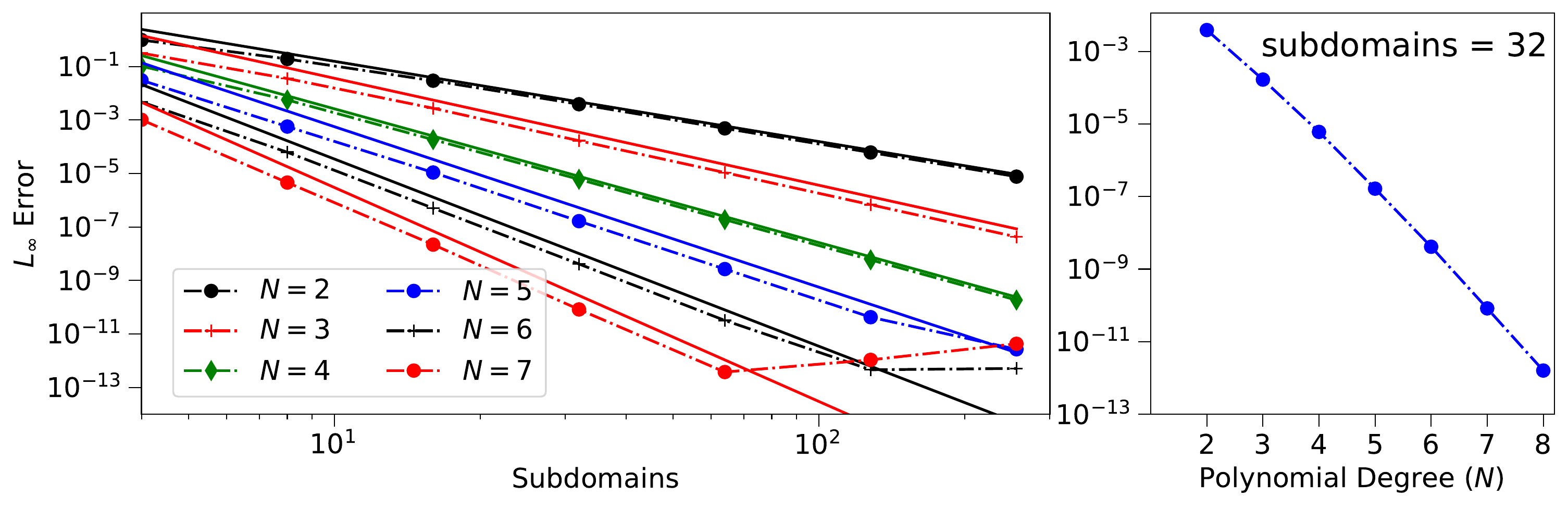}
    \end{center}
\caption{Convergence of the numerical solution for a sequence of grids for the problem setup described in Sec.~\ref{sec:exp1-dirac-field}. We consider convergence by increasing the number of elements (left) and polynomial degree (N) of the numerical approximation (right). These error profiles, computed 
as $\max_{x \in [-10,10]} \left|\bar{\psi}_{\rm exact}(10,x) - \bar{\psi}_{\rm numerical}(10,x) \right|$, are typical of our DG scheme when the solution is smooth. Notably, the error is also monitored at $x=0$. {\bf Left panel}: For a fixed value of polynomial degree, the approximation error decreases with a power law (dashed line) at a rate which closely matches the expected rate of $-(N+1)$ (solid line). For the cases $N=\{6,7\}$, round-off error effects become noticeable around $10^{-13}$. {\bf Right panel}: The DG scheme achieves exponential convergence in the approximation error as the polynomial degree is increased.}
\label{fig:prob1_6_1-dirac-field}
\end{figure}

\subsection{Wave equation with a $\delta^{(2)}(x)$ source term}
\label{sec:exp1-dirac-field}
We consider
\begin{equation}
- \partial_t^2 \psi + \partial_x^2 \psi =  \cos(t) \delta^{(2)}(x)\,,
\end{equation}
whose solution is given by Eq.~\eqref{eq:solution_s2-dirac-field}.
We will check the convergence of our numerically generated solution against
this exact solution. Before discretization, we remove some of the equation's 
singular structure by Theorem~\ref{thm:WaveMod-dirac-field}: 
let $\bar{\psi} = \psi - \cos(t) \delta(x)$ and solve
Eq.~\eqref{eq:transformed_example-dirac-field} after setting $V = a_0 = a_1 = a_3 = 0$. 
At the physical boundary points we choose fluxes that enforce simple Sommerfeld
boundary conditions and take the initial data from the exact solution.
We solve our problem on the domain $[a, b] = [-10, 10]$, set the final time T = 10,
and choose a $\Delta t$ sufficiently small such that the Runge--Kutta's timestepping error
is below the spatial discretization error. 
Despite the exact solution being both non-smooth and containing a term proportional to $\delta(x)$, Fig.~\ref{fig:prob1_6_1-dirac-field} 
shows the spectral (exponential) convergence in
the approximation. For a fixed polynomial degree $k$, the scheme's rate of convergence is observed to be $k+1$. 
Please see Ref.~\cite{field2009discontinuous} for numerical experiments with a non-zero potential.

\subsection{Persistent spurious solutions from distributional constraint violations}

\begin{figure}[!htb]
    \begin{center}
    \includegraphics[scale=0.4]{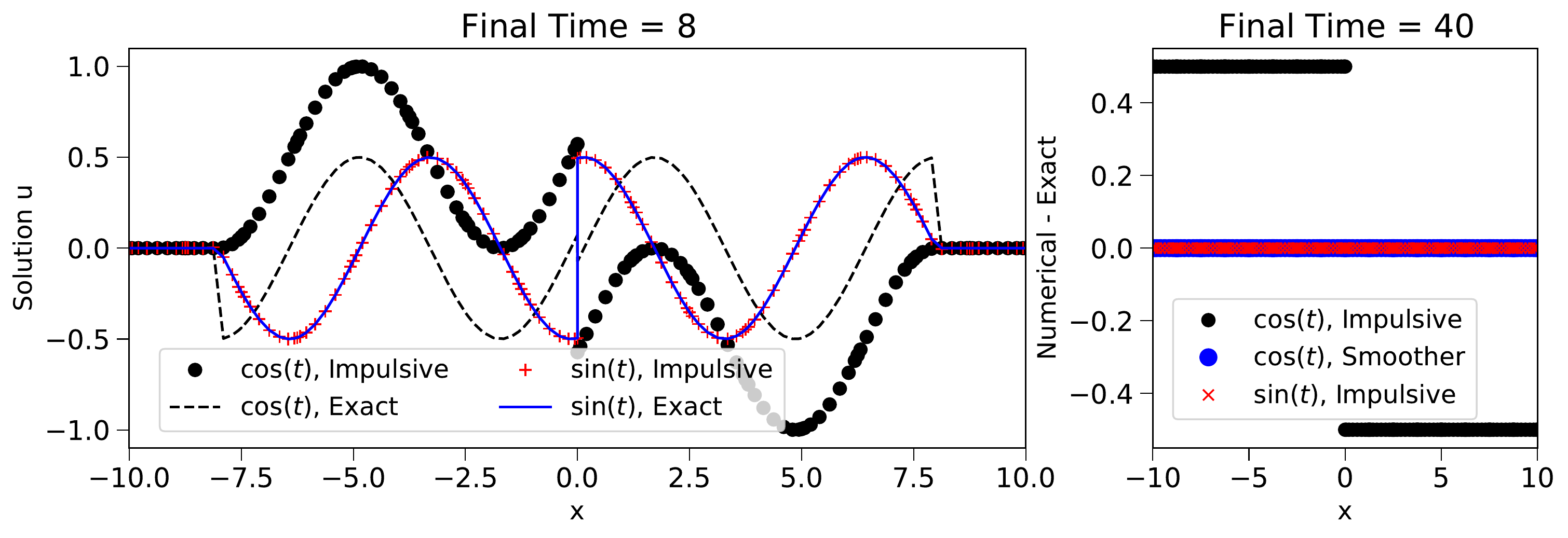}
    \end{center}
\caption{
For certain applications, the exact initial data is unknown. Here we consider possible effects on 
the numerical solution when
the initial data (here taken to be trivial data) leads to an impulsively started problem.
The left panel shows a snapshot of the numerical solution at $T=8$ where we see
that for a $\cos(t) \delta^{(1)}(x)$ source term the numerical and exact solution 
do not agree. In the right panel we plot the difference between the numerical
and exact solution and find that the 
``time-independent" spurious solution that has developed is $a_1(0) \left[H(-x)-H(x) \right] / 2$. This
is not a generic feature of the problem: no spurious solution appears 
when the source term is smoothly turned on (right figure; blue circles) or when the 
when the source term is instead taken to be $\sin(t) \delta^{(1)}(x)$.}
\label{fig:constrain_violation-dirac-field}
\end{figure}

We consider 
\begin{equation}
\label{eq:exp2_eq-dirac-field}
- \partial_t^2 \psi + \partial_x^2 \psi = a_1(t) \delta^{(1)}(x)\,,
\end{equation}
for two different cases: $a_1(t) = \sin(t)$ and $a_1(t) = \cos(t)$, the latter's
exact solution is given by Eq.~\eqref{eq:solution_s1-dirac-field}.
We solve our problem on the domain $[a, b] = [-10, 10]$ and 
we choose fluxes that enforce simple Sommerfeld
boundary conditions. When the initial data is 
taken from the exact solution the scheme's convergence properties are identical to those shown in Fig.~\ref{fig:prob1_6_1-dirac-field}.
In light of the Sec.~\ref{sec:AuxConstraints-dirac-field}'s discussion on distributional constraint violations,
we will check the for the appearance of persistent spurious solutions
when supplying trivial initial data. 

When $a_1(t) = \cos(t)$, the constraint violation is $C(t,x) = 0.5 \left[ \delta(x+t) + \delta(x-t) \right]$
and the constraint-violating spurious solution 
\begin{equation}
\bar{\psi}_{\rm CV} = \frac{1}{2} \left[H(x+t) H(-x) + H(x-t) H(x) \right] \,.
\end{equation}
is the offset that develops inside the 
future domain influenced by $(t,x)=(0,0)$. Figure \ref{fig:constrain_violation-dirac-field} (left) shows the numerical solution (black circles) offset from the exact solution (dashed black line) at $T=8$, and by $T=40$ the spurious solution has contaminated the entire computational domain (black circles).
When $a_1(t) = \sin(t)$, the constraint violation and spurious solution vanish; this too is confirmed by Fig.~\ref{fig:constrain_violation-dirac-field} (red data on both left and right panels). 
And so we see that the problematic spurious solution can be made
to vanish if it is possible to arrange 
the problem such that $F(0) = 0$.

With neither the correct initial data nor the ability to arrange $F(0) = 0$, 
a more general solution to this problem is to modify the source term
\begin{equation}\label{eq:smoothFandG-dirac-field}
F(t) \rightarrow
\left\{\begin{array}{rcl}
f(t;\tau, \delta)
F(t)
& & \text{for } 0 \leq t \leq \tau\\
F(t) & & \text{for } t > \tau,
\end{array} \right.,
\end{equation}
where $f(t;\tau, \delta) = {\textstyle \frac{1}{2}} [\mathrm{erf}(\sqrt{\delta}(t - \tau/2)+1]$
turns on the source term~\cite{field2010persistent} over the timescale $\tau$. We select 
$\tau=30$ and $\delta=.15$, which yields $f(0) \approx 10^{-16}$ and $f(t) = 1$ for $t > 30$.
Both the constraint violation and spurious solution now vanish, as shown in the right panel
of Fig.~\ref{fig:constrain_violation-dirac-field} (Blue circles).

\section{Final Remarks}
\label{sec:Final-dirac-field}

We have shown that the high--order accurate discontinuous Galerkin method developed in Ref.~\cite{field2009discontinuous} 
is applicable to the wave equation~\eqref{eq:generic-dirac-field} when written in fully first-order form~\eqref{eq:firstOrderv1-dirac-field}.
In particular, Ref.~\cite{field2009discontinuous} 
considered a wave equation with source terms of the form $a(t,x) \delta(x) + a_1(t,x) \delta{}'(x)$. In theorem~\ref{thm:WaveMod-dirac-field} we show that one can always write Eq.~\eqref{eq:generic-dirac-field} in this form, allowing for immediate application of their method to this generalized problem. The method maintains pointwise spectral convergence even at the source's location where the solution may be discontinuous, singular, or both. The numerical error has been quantified by comparing against exact
distributional solutions, and we have presented a procedure for finding particular solutions for any $\delta^{(n)}(x)$.

Our choice for writing the second-order scalar equation~\eqref{eq:generic-dirac-field} in first-order form~\eqref{eq:firstOrderv1-dirac-field}
relies on an auxiliary variable that must satisfy a distributional constraint~\eqref{eq:distributional_constraint-dirac-field}. While this constraint vanishes for all times if it does at the initial time, for many realistic problems the initial data is not known and satisfying the distributional constraint may be challenging. For trivial initial data, we show the constraint violation advects off the computational grid, leaving behind a time-independent constraint-violating spurious (or ``junk") solution in it's wake. We discuss two remedies that can be used to prevent the problematic spurious solution from appearing.

\section*{Acknowledgments}
We thank Manas Vishal for providing an independent check of Theorem 1 using Mathematica.
The authors acknowledge support of NSF Grants No. PHY-2010685 (G.K) and No. DMS-1912716 (S.F, S.G, and G.K), AFOSR Grant No. FA9550-18-1-0383 (S.G) and Office of Naval Research/Defense University Research Instrumentation Program (ONR/DURIP) Grant No. N00014181255. This material is based upon work supported by the National Science Foundation under Grant No. DMS-1439786 while a subset of the authors were in residence at the Institute for Computational and Experimental Research in Mathematics in Providence, RI, during the Advances in Computational Relativity program.

\bibliographystyle{spmpsci}


\end{document}